\documentclass[runningheads]{llncs}

\usepackage{amsmath}
\usepackage{amssymb}
\usepackage{amsfonts}
\usepackage{mathtools}
\usepackage{nicefrac}

\usepackage{enumitem}

\usepackage{algorithmic}

\usepackage{graphicx}

\usepackage[sort&compress,numbers]{natbib}

\usepackage{tikz}
\usetikzlibrary{calc}
\usetikzlibrary{decorations.pathmorphing}
\usetikzlibrary{decorations.markings}
\usetikzlibrary{fit}
\usetikzlibrary{backgrounds}
\usetikzlibrary[arrows]
\usetikzlibrary{arrows.meta}
\usetikzlibrary{
	shapes.multipart,
	matrix,
	positioning,
	shapes.callouts,
	calc}
\usepackage{pgfplots}
\pgfplotsset{compat=1.9}
\usepgfplotslibrary{statistics}
\usepgfplotslibrary{fillbetween}

\newcommand{\nats}{\mathbb{N}}

\newcommand{\reals}{\mathbb{R}}

\newcommand{\realsnonneg}{\reals_{\geq 0}}

\newcommand{\states}{\mathcal{X}}

\newcommand{\gambles}{{\mathbb{R}^{\states}}}

\newcommand{\abs}[1]{{\lvert #1\rvert}}
\newcommand{\norm}[1]{{\left\lVert #1\right\rVert}}

\newcommand{\headercite}[2]{{\hspace{1sp}{\cite[#2]{#1}}}}

\begin{document}
	
\xdefinecolor{rood}{RGB}{241, 90, 34}
\xdefinecolor{geel}{RGB}{255, 197, 11}
\xdefinecolor{roze}{RGB}{236, 0, 140}
\xdefinecolor{groen}{RGB}{179, 211, 53}
\xdefinecolor{grijs}{RGB}{72, 119, 116}

\title{Computing Expected Hitting Times for Imprecise Markov Chains\thanks{Supported by H2020-MSCA-ITN-2016 UTOPIAE, grant agreement 722734.}}
%
%
\author{Thomas Krak\orcidID{0000-0002-6182-7285}}
\authorrunning{Krak}
%
\institute{ELIS -- FLip, Ghent University, Belgium\\
\email{Thomas.Krak@UGent.be}}
\maketitle              
\begin{abstract}
We present a novel algorithm to solve a non-linear system of equations, whose solution can be interpreted as a tight lower bound on the vector of expected hitting times of a Markov chain whose transition probabilities are only partially specified. We also briefly sketch how this method can be modified to solve a conjugate system of equations that gives rise to the corresponding upper bound.
We prove the correctness of our method, and show that it converges to the correct solution in a finite number of steps under mild conditions on the system. We compare the runtime complexity of our method to a previously published method from the literature, and identify conditions under which our novel method is more efficient.

\end{abstract}
\section{Introduction}\label{sec:introduction}

We study the non-negative solution of, and in particular a numerical method to solve non-negatively, the (possibly) non-linear system
\begin{equation}\label{eq:fundamental_system}
\underline{h} = \mathbb{I}_{A^c} + \mathbb{I}_{A^c}\cdot\underline{T}\,\underline{h}\,,
\end{equation}
where $A$ is a non-empty strict subset of a non-empty finite set $\mathcal{X}$, $A^c\coloneqq \mathcal{X}\setminus A$, $\reals^\states$ is the set of all maps from $\states$ to $\reals$, $\mathbb{I}_{A^c}\in\reals^\states$ is the indicator of $A^c$, defined as $\mathbb{I}_{A^c}(x)\coloneqq 1$ if $x\in A^c$ and $\mathbb{I}_{A^c}(x)\coloneqq 0$ otherwise, $\underline{h}\in \smash{\reals^{\mathcal{X}}}$ is the vector that solves the system, and $\underline{T}:\mathbb{R}^{\mathcal{X}}\to\mathbb{R}^{\mathcal{X}}$ is a (possibly) non-linear map that satisfies the \emph{coherence} conditions~\cite{decooman:2010:markovepistemic}
\begin{enumerate}[ref={C\arabic*},label={C\arabic*}.,leftmargin=*]
\item $\underline{T}\,(\alpha f)=\alpha\underline{T}\,f$ for all $f\in\mathbb{R}^{\mathcal{X}}$ and $\alpha\in\realsnonneg$;\hfill (\emph{non-negative homogeneity})\label{coh:nonneghomogen}
\item $\underline{T}\,f + \underline{T}\,g\leq \underline{T}\,(f+g)$ for all $f,g\in\mathbb{R}^{ \mathcal{X}}$;\hfill (\emph{super-additivity})\label{coh:superadd}
\item $\min_{x\in\states}f(x)\leq\underline{T}\,f$ for all $f\in\mathbb{R}^{\mathcal{X}}$.\hfill (\emph{lower bounds})\label{coh:bounds}
\end{enumerate}
Note that in Equation~\eqref{eq:fundamental_system}, the operation $\cdot$ denotes element-wise multiplication, here applied to the functions $\mathbb{I}_{A^c}$ and $\underline{T}\,\underline{h}$.

The motivation for this problem comes from the study of Markov chains using the theory of imprecise probabilities~\cite{walley1991statistical,augustin:2014}. In this setting, $\mathcal{X}$ is interpreted as the set of possible states that some dynamical system of interest can be in, and $\underline{T}$ is the \emph{lower transition operator}, which represents the (imprecise) probability model for switching from any state $x$ to any state $y$ in a single time step, e.g.
\begin{equation}\label{eq:interpretation_lower_trans}
\bigl[\underline{T}\,\mathbb{I}_{\{y\}}\bigr](x) = \underline{P}(X_{n+1}=y\,\vert\,X_n=x)\,,
\end{equation}
where $\underline{P}(X_{n+1}=y\,\vert\,X_n=x)$ denotes the \emph{lower probability}, i.e. a tight lower bound on the probability, that the system will be in state $y$ at time $n+1$, if it is in state $x$ at time $n$. This lower bound is understood to be taken with respect to the set of stochastic processes induced by $\underline{T}$, and it is this set that is called the corresponding \emph{imprecise Markov chain}. We refer to~\cite{kozine2002interval,campos2003computing,hartfiel2006markov,skulj2006finite,decooman:2008:trees,decooman:2009:markovlimit,decooman:2010:markovepistemic,itip:stochasticprocesses,lopatatzidis2016robust} for further general information on Markov chains with imprecise probabilities. 
In this setting, the minimal non-negative solution $\underline{h}$ to Equation~\eqref{eq:fundamental_system} is then such that $\underline{h}(x)$ can be interpreted as a tight lower bound on the expected number of steps it will take the system to move from the state $x$, to any state contained in $A$. This vector $\underline{h}$ is called the (lower) expected hitting time of $A$, for the imprecise Markov chain characterised by $\underline{T}$. A version of this characterisation of the expected hitting times for imprecise Markov chains was first published by De Cooman \emph{et al.}~\cite{decooman2016impreciseprocesses}, and was later generalised to Markov chains under various imprecise probabilistic interpretations by Krak \emph{et al.}~\cite{krak2019hitting}. 

In the special case that $\underline{T}$ is a linear map, i.e. when~\ref{coh:superadd} is satisfied with equality for all $f,g\in\gambles$, then $\underline{T}$ can be interpreted as a $\lvert\states\rvert\times\lvert\states\rvert$ matrix $T$, which is called the \emph{transition matrix} in the context of (classical) Markov chains. This $T$ is row-stochastic, meaning that $\sum_{y\in\states}T(x,y)=1$ for all $x\in\states$, and $T(x,y)\geq 0$ for all $x,y\in\states$, and encodes the probability to move between states in one step, i.e. $T(x,y) = P(X_{n+1}=y\,\vert\,X_n=x)$, 
in analogy to Equation~\eqref{eq:interpretation_lower_trans}. In this case, the minimal non-negative solution $\underline{h}$ to Equation~\eqref{eq:fundamental_system} is well-known to represent the expected hitting time of $A$ for the homogenous Markov chain identified by $T$; see e.g.~\cite{norris:markovchains} for details.

Before moving on, let us first consider why the problem is somewhat ill-posed.
\begin{proposition}[\cite{krak2019hitting}]\label{prop:basic_existence}
Let $\underline{T}:\gambles\to\gambles$ be a map that satisfies the coherence conditions~\ref{coh:nonneghomogen}-\ref{coh:bounds}. Then Equation~\eqref{eq:fundamental_system} has a (unique) minimal non-negative solution $\underline{h}$ in $(\reals\cup\{+\infty\})^{\states}$, where minimality means that $\underline{h}(x)\leq h(x)$ for all $x\in\states$ and all non-negative $h\in (\reals\cup\{+\infty\})^{\states}$ that solve Equation~\eqref{eq:fundamental_system}.
\end{proposition}
So, under our present assumptions the solution is not necessarily unique---although there is a unique \emph{minimal} solution---and the solution can be infinite-valued. It can be shown that in particular $\underline{h}(x)=+\infty$ if and only if moving from $x$ to $A$ in a finite number of steps has upper probability zero. Conjugate to the notion of lower probabilities, this \emph{upper} probability is a tight upper bound on the probability that some event will happen. To exclude this case from our analysis, we will in the sequel assume that $\underline{T}$ also satisfies the \emph{reachability} condition\footnote{Here and in what follows, we take $\nats$ to be the set of natural numbers \emph{without} zero.}
\begin{enumerate}[ref={R\arabic*},label={R\arabic*}.,leftmargin=*]
\item for all $x\in A^c$, there is some $n_x\in\mathbb{N}$ such that $\bigl[\underline{T}^{n_x}\mathbb{I}_A\bigr](x)>0$\,.\label{reach_condition}
\end{enumerate}
Verifying whether a given map $\underline{T}$ satisfies~\ref{reach_condition} can be done e.g. using the algorithm for checking \emph{lower reachability} described in~\cite{DeBock:2016}, simply replacing the map $\underline{Q}$ from that work with the map $\underline{T}$. As that author notes, this algorithm takes $\smash{\mathcal{O}(\abs{\states}^3)}$ time to verify whether $\underline{T}$ satisfies~\ref{reach_condition}. We refer to Section~\ref{sec:complexity} below for details on how $\underline{T}$ may be evaluated in practice.

In any case, we are now ready to state our first main result.
\begin{proposition}\label{prop:existence_is_unique_and_finite}
Let $\underline{T}:\gambles\to\gambles$ be a map that satisfies the coherence conditions~\ref{coh:nonneghomogen}-\ref{coh:bounds} and the reachability condition~\ref{reach_condition}. Then Equation~\eqref{eq:fundamental_system} has a unique solution $\underline{h}$ in $\gambles$, and this solution is non-negative.
\end{proposition}
The proof of this result requires some setup, and is given in Section~\ref{sec:existence}. For now, we note that the unique solution $\underline{h}$ in $\gambles$ that Proposition~\ref{prop:existence_is_unique_and_finite} talks about, is also the (unique) minimal non-negative solution in $(\reals\cup\{+\infty\})^{\states}$; hence in particular, we can apply the interpretation of $\underline{h}$ as the vector of lower expected hitting times for an imprecise Markov chain, as discussed above. The following formalizes this.
\begin{corollary}\label{cor:unique_solution_is_minimal}
Let $\underline{T}:\gambles\to\gambles$ be a map that satisfies the coherence conditions~\ref{coh:nonneghomogen}-\ref{coh:bounds}. If Equation~\eqref{eq:fundamental_system} has a unique solution $\underline{h}$ in $\gambles$, and if $\underline{h}$ is non-negative, then $\underline{h}$ is the minimal non-negative solution of Equation~\eqref{eq:fundamental_system} in $(\reals\cup\{+\infty\})^{\states}$.
\end{corollary}
\begin{proof}
Let $g\in(\reals\cup\{+\infty\})^{\states}$ be the (unique) minimal non-negative solution of Equation~\eqref{eq:fundamental_system}, whose existence is guaranteed by Proposition~\ref{prop:basic_existence}. Then $g(x)\leq \underline{h}(x)$ for all $x\in\states$, because $g$ is minimal and because $\underline{h}$ is a non-negative solution to Equation~\eqref{eq:fundamental_system}. Because $\underline{h}\in\gambles$, this implies that $g(x)\in\reals$ for all $x\in\states$, whence $g\in\gambles$, and therefore it follows that $g=\underline{h}$ because $\underline{h}$ is the unique solution of Equation~\eqref{eq:fundamental_system} in $\gambles$. Hence, and because $g$ is the minimal non-negative solution of Equation~\eqref{eq:fundamental_system} in $(\reals\cup\{+\infty\})^{\states}$, so is $\underline{h}$. 
\qed
\end{proof}

Having established the existence of a unique real-valued solution to the system, our second main result is a numerical method for computing this solution. For comparison, in~\cite{krak2019hitting} the authors present an iterative method that can be directly applied as a computational tool---see Proposition~\ref{prop:solutions_by_iteration} in Section~\ref{sec:existence}---but which is only asymptotically exact, and whose runtime scales with $\lVert\underline{h}\rVert_{\infty}$, making it impractical when (some of) the expected hitting times are numerically large. In contrast, the novel method that we present in Section~\ref{sec:computational_method} is independent of $\lVert\underline{h}\rVert_{\infty}$, and we show that it converges to the correct solution in a finite number of steps under practically realistic assumptions on $\underline{T}$ (but ignoring numerical issues with finite-precision implementations).


\section{Existence of Solutions}\label{sec:existence}

The space $\gambles$ is endowed with the supremum norm, i.e. $\norm{f}\coloneqq \norm{f}_\infty\coloneqq \max_{x\in\states}\abs{f(x)}$ for all $f\in\gambles$. Mappings $M:f\mapsto Mf$ from $\gambles$ to $\gambles$ receive the induced \emph{operator norm} $\norm{M}\coloneqq \sup\{\norm{Mf}\,:\,f\in\gambles,\norm{f}\leq 1\}$. Such a map $M$ is called \emph{bounded} if it maps (norm-)bounded sets to (norm-)bounded sets; if the map is non-negatively homogeneous (i.e. if it satisfies Property~\ref{coh:nonneghomogen}) then it is bounded if and only if $\norm{M}<+\infty$. Note that this includes, as a special case, that \emph{linear} maps are bounded if and only if their norm is bounded.

An element $f\in\gambles$ that is identically equal to some $\mu\in\reals$, i.e. $f(x)=\mu$ for all $x\in\states$, is simply written as $\mu\in\gambles$. For any $f,g\in\gambles$, we take $f\leq g$ to mean that $f(x)\leq g(x)$ for all $x\in\states$. For any $B\subseteq\states$ we define the indicator $\mathbb{I}_B\in\gambles$ of $B$, for all $x\in\states$, as $\mathbb{I}_B(x)\coloneqq 1$ if $x\in B$, and $\mathbb{I}_B(x)\coloneqq 0$, otherwise.

For a given map $\underline{T}$ that satisfies~\ref{coh:nonneghomogen}-\ref{coh:bounds}, we introduce the \emph{conjugate} map $\overline{T}:\gambles\to\gambles$ that is defined, for all $f\in\gambles$, as $\overline{T}\,f\coloneqq -\underline{T}(-f)$. It is easily verified that this map satisfies the (conjugate) coherence conditions
\begin{enumerate}[ref=\upshape{UC\arabic*},label=\upshape{UC\arabic*}.,leftmargin=*]
\item $\overline{T}\,(\alpha f)=\alpha\overline{T}\,f$ for all $f\in\gambles$ and $\alpha\in\realsnonneg$;\hfill (\emph{non-negative homogeneity})\label{concoh:nonneghomogen}
\item $\overline{T}\,(f+g) \leq \overline{T}\,f + \overline{T}\,g$ for all $f,g\in\gambles$;\hfill (\emph{sub-additivity})\label{concoh:subadd}
\item $\overline{T}\,f\leq \max_{x\in\states}f(x)$ for all $f\in\gambles$.\hfill (\emph{upper bounds})\label{concoh:bounds}
\end{enumerate}
This map $\overline{T}$ is known as the \emph{upper} transition operator in the context of imprecise Markov chains, and yields, in analogy to Equation~\eqref{eq:interpretation_lower_trans}, a tight upper bound on the probability of moving between states. Note that the map $\underline{T}$ is a linear map if and only if $\underline{T}\,f=\overline{T}\,f$ for all $f\in\gambles$.
$\overline{T}$ gives rise to the system
\begin{equation}\label{eq:fundamental_system_conjugate}
\overline{h} = \mathbb{I}_{A^c} + \mathbb{I}_{A^c}\cdot\overline{T}\,\overline{h}\,,
\end{equation}
whose minimal non-negative solution $\overline{h}$ can be interpreted as a tight \emph{upper} bound on the expected hitting times of a Markov chain with imprecise probabilities. It follows that in particular $\underline{h}\leq \overline{h}$; we refer to~\cite{decooman2016impreciseprocesses,krak2019hitting} for further details.

The following properties are well-known to hold for maps that satisfy~\ref{coh:nonneghomogen}-\ref{coh:bounds}; we state them here for convenience. Reference~\cite{krak2019extendedhitting} happens to contain all of these, but most of the references on imprecise Markov chains state at least some of them. Here and in what follows, for any $n\in\nats$, we use $M^n$ to denote the $n$-fold composition of a map $M:\gambles\to\gambles$ with itself.
\begin{lemma}
Let $\underline{T}:\gambles\to\gambles$ be a map that satisfies~\ref{coh:nonneghomogen}-\ref{coh:bounds}, and let $\overline{T}$ be its conjugate map. Then, for all $f,g\in\gambles$ and all $n\in\nats$, it holds that
\begin{enumerate}[ref=\upshape{T\arabic*},label=\upshape{T\arabic*}.,leftmargin=*]
	\item $\min_{x\in\states}f(x) \leq \underline{T}^nf\leq \overline{T}^nf\leq \max_{x\in\states} f(x)$;\hfill (\emph{bounds})\label{transprop:bounds}
	\item $f\leq g \Rightarrow \underline{T}^nf\leq \underline{T}^ng$\hfill (\emph{monotonicity})\label{transprop:monotone}
	\item $\underline{T}^n(f+\mu) = \underline{T}^nf + \mu$ for all constant $\mu\in\gambles$\hfill (\emph{constant additivity})\label{transprop:constant_additive}
	\item $\norm{\underline{T}^nf-\underline{T}^ng}\leq \norm{f-g}$\hfill (\emph{non-expansiveness})\label{transprop:nonexpand}
\end{enumerate}
\end{lemma}
\begin{corollary}\label{cor:lower_trans_is_bounded}
Let $\underline{T}:\gambles\to\gambles$ be a map that satisfies~\ref{coh:nonneghomogen}-\ref{coh:bounds}. Then $\underline{T}$ is bounded.
\end{corollary}
\begin{proof}
Let $f\in\gambles$ be such that $\norm{f}\leq 1$, and fix any $x\in\states$. It then follows from Property~\ref{transprop:bounds} and the definition of $\norm{f}$ that $-\norm{f} \leq [\underline{T}\,f](x) \leq \norm{f}$, and hence $\abs{[\underline{T}\,f](x)}\leq \norm{f}\leq 1$. Because this is true for any $x\in\states$, it follows that $\norm{\underline{T}\,f}\leq 1$. Because this is true for any $f\in\gambles$ with $\norm{f}\leq 1$, it follows from the definition of the operator norm that $\norm{\underline{T}}\leq 1$, whence $\underline{T}$ is bounded.
\qed
\end{proof}

The following result provides an iterative method to find the minimal non-negative solution(s) to the, in general, non-linear systems~\eqref{eq:fundamental_system} and~\eqref{eq:fundamental_system_conjugate}. 
\begin{proposition}[\cite{krak2019hitting}]\label{prop:solutions_by_iteration}
Let $\underline{T}:\gambles\to\gambles$ be a map that satisfies~\ref{coh:nonneghomogen}-\ref{coh:bounds}, and let $\overline{T}$ be its conjugate map. Define $\underline{h}_0\coloneqq \overline{h}_0\coloneqq \mathbb{I}_{A^c}$ and, for all $n\in\nats$, let $\underline{h}_n\coloneqq \mathbb{I}_{A^c}+\mathbb{I}_{A^c}\cdot\underline{T}\,\underline{h}_{n-1}$ and $\overline{h}_n\coloneqq \mathbb{I}_{A^c}+\mathbb{I}_{A^c}\cdot\overline{T}\,\overline{h}_{n-1}$. Then $\underline{h}_*\coloneqq \lim_{n\to+\infty} \underline{h}_n$ and $\overline{h}_*\coloneqq \lim_{n\to+\infty} \overline{h}_n$ exist in $(\reals\cup\{+\infty\})^{\states}$. Moreover, $\underline{h}_*$ is the minimal non-negative solution to Equation~\eqref{eq:fundamental_system}, and $\overline{h}_*$ is the minimal non-negative solution to Equation~\eqref{eq:fundamental_system_conjugate}.
\end{proposition}
\begin{proof}
This follows from~\cite[Proposition 10, Theorem 12, and Corollary 13]{krak2019hitting}.\qed
\end{proof}
The scheme proposed in this result can also be directly applied as a computational method, the complexity of which we will analyse and compare to our proposed method, in Section~\ref{sec:complexity}.

We need the following, somewhat abstract, definition, that is often encountered in the context of Markov chains with imprecise probabilities.
\begin{definition}
Let $\mathcal{T}$ be a set of maps from $\gambles$ to $\gambles$. For any map $T:\gambles\to\gambles$ and any $x\in\states$, let $T_x:\gambles\to\reals:f\mapsto[Tf](x)$. Then we say that $\mathcal{T}$ has \emph{separately specified rows}, if $\mathcal{T} = \bigl\{T:\gambles\to\gambles\,\big\vert\,\forall x\in\states: T_x\in\mathcal{T}_x\bigr\}$ where, for all $x\in\states$, $\mathcal{T}_x\coloneqq \{T_x\,\vert\,T\in\mathcal{T}\}$.
\end{definition}
This definition is needed to obtain the following well-known dual representation of $\underline{T}$, of which we will make extensive use throughout the remainder of this work. 
Because the result is well-known (see e.g.~\cite{decooman:2009:markovlimit,itip:stochasticprocesses}), we omit the proof here.
\begin{proposition}\label{prop:dual_set_representation}
Let $\underline{T}:\gambles\to\gambles$ be a map that satisfies~\ref{coh:nonneghomogen}-\ref{coh:bounds}. Then there is a unique non-empty, closed, and convex set $\mathcal{T}$ that has separately specified rows, such that each $T\in\mathcal{T}$ is a linear map from $\gambles$ to $\gambles$ that satisfies~\ref{coh:nonneghomogen}-\ref{coh:bounds}, and such that, for all $f\in\gambles$, $\underline{T}\,f=\inf_{T\in\mathcal{T}}Tf$. Moreover, for all $f\in\gambles$ there is some $T\in\mathcal{T}$ such that $\underline{T}\,f=Tf$.
\end{proposition}
Note that the elements $T\in\mathcal{T}$ of the dual representation $\mathcal{T}$ of $\underline{T}$ are linear maps from $\gambles$ to $\gambles$ that all satisfy \ref{coh:nonneghomogen}-\ref{coh:bounds}. Therefore, as noted in Section~\ref{sec:introduction}, these maps can be represented as $\lvert\states\rvert\times \lvert\states\rvert$ matrices that are row-stochastic.

The dual representation $\mathcal{T}$ of $\underline{T}$ is also conveniently connected to the conjugate map $\overline{T}$, as follows. We omit the proof, which is straightforward.
\begin{corollary}
Let $\underline{T}:\gambles\to\gambles$ be a map that satisfies~\ref{coh:nonneghomogen}-\ref{coh:bounds}, let $\overline{T}$ be its conjugate map, and let $\mathcal{T}$ be its dual representation. Then for all $f\in\gambles$ it holds that $\overline{T}\,f=\sup_{T\in\mathcal{T}}Tf$ and $\overline{T}\,f=Tf$ for some $T\in\mathcal{T}$.
\end{corollary}

We note that, because the elements $T\in\mathcal{T}$ of the dual representation $\mathcal{T}$ of $\underline{T}$ are linear maps from $\gambles$ to $\gambles$ that satisfy \ref{coh:nonneghomogen}-\ref{coh:bounds}, the minimal non-negative solutions $h_T$ of the linear system $h_T=\mathbb{I}_{A^c}+\mathbb{I}_{A^c}\cdot Th_T$ exist due to Proposition~\ref{prop:basic_existence}. As established in Section~\ref{sec:introduction}, these $h_T$ can be interpreted as the vectors of expected hitting times of Markov chains parameterised by $T\in\mathcal{T}$.

The next result is established in~\cite{krak2019hitting}, and formalises the interpretation of the solutions to the systems~\eqref{eq:fundamental_system} and~\eqref{eq:fundamental_system_conjugate} as providing bounds on the expected hitting times for the set of Markov chains induced by the elements $T$ of $\mathcal{T}$. 
\begin{proposition}[\cite{krak2019hitting}]\label{prop:solutions_are_bounds_and_reached}
Let $\underline{T}:\gambles\to\gambles$ be a map that satisfies~\ref{coh:nonneghomogen}-\ref{coh:bounds}, let $\overline{T}$ be its conjugate map, and let $\mathcal{T}$ be its dual representation. For any $T\in\mathcal{T}$, let $h_T$ be the minimal non-negative solution of the linear system $h_T=\mathbb{I}_{A^c}+\mathbb{I}_{A^c}\cdot Th_T$. Let $\underline{h}$ and $\overline{h}$ be the minimal non-negative solutions of the systems~\eqref{eq:fundamental_system} and~\eqref{eq:fundamental_system_conjugate}, respectively. Then
\begin{equation*}
\underline{h} = \inf_{T\in\mathcal{T}}h_T\quad\text{and}\quad \overline{h}=\sup_{T\in\mathcal{T}}h_T\,.
\end{equation*}
Moreover, there is some $T\in\mathcal{T}$ such that $\underline{h}=h_T$, and there is some (possibly different) $T\in\mathcal{T}$ such that $\overline{h}=h_T$.
\end{proposition}
\begin{proof}
This follows from~\cite[Lemma 8, Theorem 12 and Corollary 13]{krak2019hitting}.\qed
\end{proof}

We are now ready to begin the analysis that shows that the reachability condition~\ref{reach_condition} is a sufficient assumption to establish Proposition~\ref{prop:existence_is_unique_and_finite}.
We start with the following result, which says that if $\underline{T}$ satisfies the reachability condition~\ref{reach_condition}, then so do the elements of its dual representation $\mathcal{T}$.
\begin{lemma}\label{lemma:lower_reachability_is_precise_reachability}
Let $\underline{T}:\gambles\to\gambles$ be a map that satisfies~\ref{coh:nonneghomogen}-\ref{coh:bounds} and the reachability condition~\ref{reach_condition}, and let $\mathcal{T}$ be its dual representation. Then for all $T\in\mathcal{T}$ and all $x\in A^c$, there is some $n_x\in\nats$ such that $\bigl[T^{n_x}\mathbb{I}_{A}\bigr](x)>0$; hence $T$ then also satisfies~\ref{reach_condition}.
\end{lemma}
\begin{proof}
Fix any $T\in\mathcal{T}$ and any $x\in A^c$. Due to~\ref{reach_condition}, there is some $n_x\in\nats$ such that $[\underline{T}^{n_x}\mathbb{I}_{A}](x)>0$. Because $T\in\mathcal{T}$, it follows from Proposition~\ref{prop:dual_set_representation} that
\begin{equation*}
[T^{n_x}\mathbb{I}_{A}](x) = [TT^{n_x-1}\mathbb{I}_{A}](x) \geq \left[\underline{T}T^{n_x-1}\mathbb{I}_A\right](x) \geq \cdots \geq [\underline{T}^{n_x}\mathbb{I}_{A}](x)>0\,,
\end{equation*}
where we repeatedly used the monotonicity property~\ref{transprop:monotone}.
\qed
\end{proof}
Let us investigate this reachability property for linear maps more fully. We need the following two results.
\begin{lemma}\label{lemma:first_step_reachable}
Let $T:\gambles\to\gambles$ be a linear map that satisfies~\ref{coh:nonneghomogen}-\ref{coh:bounds} and the reachability condition~\ref{reach_condition}. Then there is some $x\in A^c$ such that $[T\mathbb{I}_{A}](x)>0$.
\end{lemma}
\begin{proof}
Suppose \emph{ex absurdo} that this is false. Since it follows from Property~\ref{transprop:bounds} that $[T\mathbb{I}_{A}](x)\geq \min_{y\in\states}\mathbb{I}_{A}(y)=0$, we must have $[T\mathbb{I}_{A}](x)=0$ for all $x\in A^c$. This provides the induction base for $n=1$ in the following induction argument: suppose that for some $n\in\nats$, $[T^n\mathbb{I}_A](x)=0$ for all $x\in A^c$; we will show that then also $[T^{n+1}\mathbb{I}_A](x)=0$ for all $x\in A^c$.

First, for any $x\in A$ we have $[T^n\mathbb{I}_A](x)\leq \max_{y\in\states}\mathbb{I}_{A}(y)=1=\mathbb{I}_A(x)$ due to~\ref{transprop:bounds}. Moreover, for any $x\in A^c$ we have $[T^n\mathbb{I}_A](x)=0=\mathbb{I}_A(x)$ by the induction hypothesis. Hence we have $T^n\mathbb{I}_A\leq \mathbb{I}_A$. It follows that, for any $x\in A^c$,
\begin{equation*}
[T^{n+1}\mathbb{I}_A](x)=[TT^{n}\mathbb{I}_A](x)\leq [T\mathbb{I}_A](x)=0\,,
\end{equation*}
using the monotonicity of $T$ (Property~\ref{transprop:monotone}) for the inequality, and the argument at the beginning of this proof for the final equality. Because it follows from Property~\ref{transprop:bounds} that $[T^{n+1}\mathbb{I}_A](x)\geq \min_{y\in\states}\mathbb{I}_A(y)=0$, this implies that $[T^{n+1}\mathbb{I}_A](x)=0$ for all $x\in A^c$. This concludes the proof of the induction step.

Hence we have established that, for all $x\in A^c$, $[T^{n}\mathbb{I}_A](x)=0$ for all $n\in\nats$ which, because $A^c$ is non-empty, contradicts the assumption that $T$ satisfies R1. Hence, our assumption must be false, and there must be some $x\in A^c$ such that  $[T\mathbb{I}_{A}](x)>0$.
\qed
\end{proof}

\begin{lemma}\label{lemma:next_step_reachable_means_earlier_step_reachable}
Let $T:\gambles\to\gambles$ be a linear map that satisfies~\ref{coh:nonneghomogen}-\ref{coh:bounds} and the reachability condition~\ref{reach_condition}. Fix any $x\in A^c$ and $n\in\nats$ with $n>1$, and suppose that  $\bigl[T^n\mathbb{I}_{A}\bigr](x)>0$ and $\bigl[T^m\mathbb{I}_{A}\bigr](x)=0$ for all $m\in\nats$ with $m<n$. Then there is some $y\in A^c$ such that $[T\mathbb{I}_{\{y\}}](x)>0$ and $[T^{n-1}\mathbb{I}_A](y)>0$.
\end{lemma}
\begin{proof}
Because $\mathbb{I}_A=\sum_{z\in A}\mathbb{I}_{\{z\}}$, and using the linear character of $T$---and therefore of $T^n$---we have
\begin{equation*}
0< \bigl[T^n\mathbb{I}_{A}\bigr](x) = \sum_{z\in A}[T^n\mathbb{I}_{\{z\}}](x)\,,
\end{equation*}
and hence there must be some $z\in A$ such that $[T^n\mathbb{I}_{\{z\}}](x)>0$. Next, we note that for any $f\in\gambles$ it holds that $f=\sum_{y\in\states}f(y)\mathbb{I}_{\{y\}}$. Hence, expanding the product $T^n$ and using the linearity of $T$---and therefore of $T^{n-1}$---yields
\begin{equation*}
[T^n\mathbb{I}_{\{z\}}](x)=\bigl[T\bigl(T^{n-1}\mathbb{I}_{\{z\}}\bigr)\bigr](x)=\sum_{y\in\states}\Bigl[T\Bigl([T^{n-1}\mathbb{I}_{\{z\}}](y)\mathbb{I}_{\{y\}}\Bigr)\Bigr](x)\,.
\end{equation*}
Because $[T^n\mathbb{I}_{\{z\}}](x)>0$, there must be some $y\in\states$ such that 
\begin{equation*}
0<\Bigl[T\Bigl([T^{n-1}\mathbb{I}_{\{z\}}](y)\mathbb{I}_{\{y\}}\Bigr)\Bigr](x) = [T^{n-1}\mathbb{I}_{\{z\}}](y)\bigl[T\mathbb{I}_{\{y\}}\bigr](x)\,,
\end{equation*} 
using the linearity of $T$ for the equality. Since both factors in this expression are clearly non-negative due to Property~\ref{transprop:bounds}, this implies that $[T\mathbb{I}_{\{y\}}](x)>0$ and $[T^{n-1}\mathbb{I}_{\{z\}}](y)>0$. Because $z\in A$ we have $\mathbb{I}_{\{z\}}\leq\mathbb{I}_A$, and hence this last inequality together with Property~\ref{transprop:monotone} implies that
\begin{equation*}
0 < [T^{n-1}\mathbb{I}_{\{z\}}](y) \leq [T^{n-1}\mathbb{I}_A](y)\,,
\end{equation*}
so we see that it only remains to show that $y\in A^c$. Suppose \emph{ex absurdo} that $y\in A$. Then $\mathbb{I}_{\{y\}}\leq\mathbb{I}_A$, and because $[T\mathbb{I}_{\{y\}}](x)>0$, it follows from Property~\ref{transprop:monotone} that
\begin{equation*}
0<[T\mathbb{I}_{\{y\}}](x)\leq [T\mathbb{I}_A](x)\,,
\end{equation*}
which contradicts the assumption that $[T^m\mathbb{I}_A](x)=0$ for all $m<n$. Hence $y\in A^c$, which concludes the proof.
\qed
\end{proof}

In the sequel, we will be interested in some results about mappings on subspaces of $\gambles$. The following definition introduces the required notation.
\begin{definition}\label{def:restriction}
For any $f\in\gambles$, we denote by $f\vert_{A^c}$ the restriction of $f$ to $A^c$, i.e. the mapping $f\vert_{A^c}:A^c\to\reals:x\mapsto f(x)$, which is an element of $\mathbb{R}^{A^c}$.

Moreover, let $M:\gambles\to\gambles$ be a map. Then we define its restriction $M\vert_{A^c}$ to the subspace $\mathbb{R}^{A^c}$ of $\gambles$, for all $f\in\mathbb{R}^{A^c}$, as
\begin{equation*}
M\vert_{A^c}(f) \coloneqq \bigr(M(g\cdot\mathbb{I}_{A^c})\bigl)\vert_{A^c}\quad(g\in\gambles:g\vert_{A^c}=f)\,.
\end{equation*}
\end{definition}
The space $\mathbb{R}^{A^c}$ again receives the supremum norm, and maps from $\mathbb{R}^{A^c}$ to $\mathbb{R}^{A^c}$ the induced operator norm. Note that if the map $M$ in the previous definition is a linear map, then also its restriction $M\vert_{A^c}$ is a linear map. Hence in particular, in that case $M$ can be interpreted as a matrix, and the restriction $M\vert_{A^c}$ can then be interpreted as the $\lvert A^c\rvert\times\lvert A^c\rvert$ submatrix of $M$ on the coordinates $A^c$. Moreover, we note the following simple property:
\begin{lemma}\label{lemma:restriction_bounded_is_bounded}
Let $M:\gambles\to\gambles$ be a bounded map. Then its restriction $M\vert_{A^c}$ to $\mathbb{R}^{A^c}$ is also bounded.
\end{lemma}
\begin{proof}
Choose any $f\in\reals^{A^c}$ with $\norm{f}\leq 1$, and let $g\in\gambles$ be such that $g\vert_{A^c}=f$ and $g(x)=0$ for all $x\in A$. Then it holds that $g\cdot\mathbb{I}_{A^c}=g$ and $\norm{g}=\norm{f}$. Hence it follows that
\begin{equation*}
\norm{M\vert_{A^c}f}=\norm{\bigl(M(g\cdot\mathbb{I}_{A^c})\bigr)\big\vert_{A^c}} = \norm{(Mg)\vert_{A^c}} \leq \norm{Mg}\leq \norm{M}\,,
\end{equation*}
where the last inequality used that $\norm{g}=\norm{f}\leq 1$.
\qed
\end{proof}

We next need some elementary results from the spectral theory of bounded linear maps, in particular for linear maps from $\reals^{A^c}$ to $\reals^{A^c}$. The following definition aggregates concepts from~\cite[Chapter 7, Definitions 1.2, 3.1, 3.5]{dunford1988linear}.
\begin{definition}[\cite{dunford1988linear}]
Let $M:\reals^{A^c}\to\reals^{A^c}$ be a bounded linear map. The \emph{spectrum} of $M$ is the set $\sigma(M)\coloneqq \{\lambda\in\mathbb{C}\,:\,(\lambda I-M)\text{ is not one-to-one}\}$, where $I$ denotes the identity map on $\reals^{A^c}$. The \emph{spectral radius of $M$} is $\rho(M)\coloneqq \sup_{\lambda\in\sigma(M)}\abs{\lambda}$. The \emph{resolvent $R(\lambda, M)$ of $M$} is defined for all $\lambda\in\mathbb{C}\setminus \sigma(M)$ as $R(\lambda,M)\coloneqq (\lambda I-M)^{-1}$.
\end{definition}
Because $A^c$ is finite (since $\states\supset A^c$ is finite), the spectrum $\sigma(M)$ of any bounded linear map from $\reals^{A^c}$ to $\reals^{A^c}$ is a non-empty and finite set~\cite[Chapter 7, Corollary 1.4]{dunford1988linear}. 
We will need the following properties of these objects.
\begin{lemma}[\headercite{dunford1988linear}{Chapter 7, Lemma 3.4}]\label{lemma:resolvent_is_series}
Let $M:\reals^{A^c}\to\reals^{A^c}$ be a bounded linear map. Then for all $\lambda\in\mathbb{C}$ with $\abs{\lambda}>\rho(M)$, the series $\sum_{k=0}^{+\infty}\frac{1}{\lambda^{k+1}}M^k$ converges in norm to $R(\lambda, M)$.
\end{lemma}
The next property is well-known, but we had some trouble finding an easy-to-use reference; hence we prove it explicitly below.
\begin{corollary}\label{cor:powers_vanish}
Let $M:\reals^{A^c}\to\reals^{A^c}$ be a bounded linear map, and suppose that $\rho(M)<1$. Then $\lim_{n\to+\infty}M^n=0$.
\end{corollary}
\begin{proof}
Because $\rho(M)<1$, we have $\lim_{n\to+\infty}\norm{\sum_{k=0}^n M^k - R(1,M)} = 0$ by  Lemma~\ref{lemma:resolvent_is_series}.
Now for any $n\in\nats$, it holds that
\begin{align*}
\norm{M^n}  = \norm{\sum_{k=0}^{n} M^k - \sum_{k=0}^{n-1} M^k} &= \norm{\sum_{k=0}^{n} M^k - R(1,M) + R(1,M) - \sum_{k=0}^{n-1} M^k} \\
 &\leq \norm{\sum_{k=0}^{n} M^k - R(1,M)} + \norm{R(1,M) - \sum_{k=0}^{n-1} M^k}\,.
\end{align*}
Because both summands on the right-hand side vanish as we take $n$ to $+\infty$, it follows that $\lim_{n\to+\infty}\norm{M^n}=0$, or in other words, that $\lim_{n\to+\infty}M^n=0$.
\qed
\end{proof}

\begin{lemma}[\headercite{dunford1988linear}{Chapter 7, Lemma 3.4}]\label{lemma:bound_spectral_radius}
Let $M:\reals^{A^c}\to\reals^{A^c}$ be a bounded linear map. Then $\rho(M)\leq \norm{M^n}^{\frac{1}{n}}$ for all $n\in\nats$.
\end{lemma}
The crucial observation is now that, for a linear map $T$ that satisfies~\ref{coh:nonneghomogen}-\ref{coh:bounds} and the reachability condition~\ref{reach_condition}, its restriction $T\vert_{A^c}$ to $\reals^{A^c}$ is a bounded linear map (due to Corollary~\ref{cor:lower_trans_is_bounded} and Lemma~\ref{lemma:restriction_bounded_is_bounded}) that for large enough $n\in\nats$, as we will see in Corollary~\ref{cor:norm_restriction_eventually_small_enough}, satisfies $\norm{(T\vert_{A^c})^n}<1$. We can therefore use Lemma~\ref{lemma:bound_spectral_radius} to establish that the spectral radius of $T\vert_{A^c}$ is less than 1, and we can then apply Lemma~\ref{lemma:resolvent_is_series} and Corollary~\ref{cor:powers_vanish}.

Let us establish that these claims are indeed true. We start with the following result, which gives some basic properties of these restrictions.
\begin{lemma}\label{lemma:basic_properties_restriction}
Let $T:\gambles\to\gambles$ be a linear map that satisfies~\ref{coh:nonneghomogen}-\ref{coh:bounds}, and let $T\vert_{A^c}$ be its restriction to $\reals^{A^c}$, as in Definition~\ref{def:restriction}. Then $(T\vert_{A^c})^nf \leq (T\vert_{A^c})^ng$ for all $n\in\nats$ and all $f,g\in\reals^{A^c}$ with $f\leq g$. Moreover, it holds that $0\leq (T\vert_{A^c})^n1\leq 1$ for all $n\in\nats$.
\end{lemma}
\begin{proof}
For the first claim, we proceed by induction on $n$. For the induction base, let $f',g'\in\gambles$ be such that $f'(x)\coloneqq 0$ and $g'(x)\coloneqq 0$ for all $x\in A$, and $f'(x)\coloneqq f(x)$ and $g'(x)\coloneqq g(x)$ for all $x\in A^c$. 
Then, clearly, $f'\vert_{A^c}=f$ and $g'\vert_{A^c}=g$, and $f'\leq g'$ because $f\leq g$. Moreover, it holds that $f'\cdot\mathbb{I}_{A^c}=f'$ and $g'\cdot\mathbb{I}_{A^c}=g'$.
It follows from Definition~\ref{def:restriction} that, for any $x\in A^c$, since $f'\leq g'$, it holds that
\begin{align*}
[T\vert_{A^c}f](x) = [T(f'\cdot\mathbb{I}_{A^c})](x) &= [T(f')](x) \\
&\leq [T(g')](x) = [T(g'\cdot\mathbb{I}_{A^c})](x) = [T\vert_{A^c}g](x)\,,
\end{align*}
where we used that $T$ satisfies Property~\ref{transprop:monotone} for the inequality. Because this is true for any $x\in\states$, it follows that $T\vert_{A^c}f\leq T\vert_{A^c}g$.

Now suppose that $(T\vert_{A^c})^nf\leq (T\vert_{A^c})^ng$ for some $n\in\nats$. Then also
\begin{equation*}
(T\vert_{A^c})^{n+1}f = T\vert_{A^c}\bigl((T\vert_{A^c})^nf\bigr) \leq T\vert_{A^c}\bigl((T\vert_{A^c})^ng\bigr) = (T\vert_{A^c})^{n+1}g\,,
\end{equation*}
using the argument for the induction base for the inequality.

For the second claim, start by noting that $\mathbb{I}_{A^c}\vert_{A^c}=1$, $\mathbb{I}_{A^c}\cdot\mathbb{I}_{A^c}=\mathbb{I}_{A^c}$, and $0\leq \mathbb{I}_{A^c}\leq 1$. Hence it follows from Definition~\ref{def:restriction} that, for all $x\in A^c$,
\begin{equation*}
0=[T0](x)\leq [T\mathbb{I}_{A^c}](x) = [T\vert_{A^c}1](x) = [T\mathbb{I}_{A^c}](x) \leq [T1](x)=1\,,
\end{equation*}
where we used the linearity of $T$ for the first equality, Property~\ref{transprop:monotone} for the first inequality, Definition~\ref{def:restriction} for the second and third equalities, Property~\ref{transprop:monotone} for the second inequality, and Property~\ref{transprop:bounds} for the last equality.

Hence we have $0\leq T\vert_{A^c}1\leq 1$. By repeatedly using the monotonicity property established in the first part of this proof, it follows that also for all $n>1$ we have that
\begin{align*}
(T\vert_{A^c})^n1 = (T\vert_{A^c})^{n-1}\bigl(T\vert_{A^c}1\bigr) &\leq (T\vert_{A^c})^{n-1}1 \\
 &= (T\vert_{A^c})^{n-2}\bigl(T\vert_{A^c}1\bigr) \leq\cdots \leq T\vert_{A^c}1\leq 1\,.
\end{align*}
To establish that $0\leq (T\vert_{A^c})^n1$ for all $n\in\nats$, we note that the case for $n=1$ was already given above. Now, for any $n>1$ we have
\begin{equation*}
(T\vert_{A^c})^n1=(T\vert_{A^c})^{n-1}(T\vert_{A_c}1)\geq (T\vert_{A^c})^{n-1}0=0\,,
\end{equation*}
where for the inequality we used the property $0\leq T\vert_{A^c}1$ together with the monotonicity property established in the first part of this proof; and where we used the linearity of $(T\vert_{A^c})^{n-1}$ for the final equality.
\qed
\end{proof}

\begin{lemma}\label{lemma:complement_eventually_moves}
Let $T:\gambles\to\gambles$ be a linear map that satisfies~\ref{coh:nonneghomogen}-\ref{coh:bounds} and the reachability condition~\ref{reach_condition}. Then there is some $n\in\nats$ such that for all $k\in\nats$ with $k\geq n$, it holds that $[(T\vert_{A^c})^k1](x)<1$ for all $x\in A^c$.
\end{lemma}
\begin{proof}
Because $T$ satisfies~\ref{reach_condition}, for all $x\in A^c$ there is some unique \emph{minimal} $n_x\in\nats$ such that $[T^{n_x}\mathbb{I}_{A}](x)>0$. Let $n\coloneqq \max_{x\in\states} n_x$. We will show that $[(T\vert_{A^c})^k1](x)<1$ for all $x\in\states$, for all $k\in\nats$ with $k\geq n$.

We proceed by induction, as follows. Fix any $x\in A^c$ with $n_x=1$; then $[T\mathbb{I}_A](x)>0$. Because $1=\mathbb{I}_A+\mathbb{I}_{A^c}$, and since $[T1](x)=1$ due to Property~\ref{transprop:bounds}, it therefore follows from the fact that $T$ is a linear map that
\begin{equation*}
1=[T1](x)=[T\mathbb{I}_A](x)+[T\mathbb{I}_{A^c}](x)>[T\mathbb{I}_{A^c}](x)\,.
\end{equation*}
With $g\coloneqq\mathbb{I}_{A^c}$ we have $g\vert_{A^c}=1$ and $g\cdot\mathbb{I}_{A^c}=\mathbb{I}_{A^c}$, so it follows from the above together with Definition~\ref{def:restriction} that
\begin{equation*}
[T\vert_{A^c}1](x) = [T\mathbb{I}_{A^c}](x) < 1\,.
\end{equation*}
Hence for any $x\in A^c$ with $n_x=1$ it holds that $[T\vert_{A^c}1](x)<1$. This enables the induction base for $m=1$ in the following induction argument: suppose that for some $m\in\nats$ it holds that $[(T\vert_{A^c})^{m}1](x)<1$ for all $x\in A^c$ with $n_x\leq m$; we will show that then also $[(T\vert_{A^c})^{m+1}1](x)<1$ for all $x\in A^c$ with $n_x\leq m+1$.

First consider any $x\in A^c$ with $n_x\leq m$. By Lemma~\ref{lemma:basic_properties_restriction} we have that $(T\vert_{A^c})1\leq 1$ and therefore, that
\begin{equation*}
[(T\vert_{A^c})^{m+1}1](x)=[(T\vert_{A^c})^{m}(T\vert_{A^c}1)](x) \leq [(T\vert_{A^c})^{m}1](x)<1\,,
\end{equation*}
where we used the monotonicity of $(T\vert_{A^c})^{m}$ established in Lemma~\ref{lemma:basic_properties_restriction} for the first inequality, and the induction hypothesis that $[(T\vert_{A^c})^{m}1](x)<1$ since $n_x\leq m$ for the final inequality.
Hence we have $[(T\vert_{A^c})^{m+1}1](x)\leq [(T\vert_{A^c})^{m}1](x)<1$ for all $x\in A^c$ with $n_x\leq m$.

Now consider any $x\in A^c$ with $n_x=m+1$. Because $n_x$ was chosen to be minimal, this means that $[T^{m+1}\mathbb{I}_A](x)>0$ and $[T^k\mathbb{I}_A](x)=0$ for all $k\in\nats$ with $k\leq m$. Due to Lemma~\ref{lemma:next_step_reachable_means_earlier_step_reachable}, and since $m+1>1$, this implies that there is some $y\in A^c$ such that $[T\mathbb{I}_{\{y\}}](x)>0$ and $[T^{m}\mathbb{I}_A](y)>0$. This last inequality implies that $n_y\leq m$, since $n_y$ was chosen to be minimal. By the induction hypothesis, this implies that $[(T\vert_{A^c})^{m}1](y)<1$.

Let $g\in\gambles$ be such that  $g\vert_{A^c}=(T\vert_{A^c})^{m}1$, and $g(z)\coloneqq 0$ for all $z\in A$. Then $g\cdot\mathbb{I}_{A^c}=g$, whence it follows from Definition~\ref{def:restriction} that
\begin{align*}
[(T\vert_{A^c})^{m+1}1](x) &= \bigl[T\vert_{A^c}\bigl((T\vert_{A^c})^{m}1\bigr)\bigr](x) = \bigl[T(g\cdot\mathbb{I}_{A^c})\bigr](x) = \bigl[Tg\bigr](x)\,.
\end{align*}
Next, we note that $g=\sum_{z\in\states}g(z)\mathbb{I}_{\{z\}}$, and so, using the linear character of $T$, it holds that $\bigl[Tg\bigr](x)= \sum_{z\in\states} g(z)\bigl[T\mathbb{I}_{\{z\}}](x)$. Focussing on the summand for $z=y$ we find that
\begin{align*}
\sum_{z\in\states} g(z)\bigl[T\mathbb{I}_{\{z\}}](x) &= g(y)[T\mathbb{I}_{\{y\}}](x) + \sum_{z\in\states\setminus\{y\}} g(z)\bigl[T\mathbb{I}_{\{z\}}\bigr](x) \\
 &< [T\mathbb{I}_{\{y\}}](x)+ \sum_{z\in\states\setminus\{y\}} g(z)\bigl[T\mathbb{I}_{\{z\}}\bigr](x) \\
 &\leq [T\mathbb{I}_{\{y\}}](x)+ \sum_{z\in\states\setminus\{y\}} \bigl[T\mathbb{I}_{\{z\}}\bigr](x) = \sum_{z\in\states} \bigl[T\mathbb{I}_{\{z\}}\bigr](x) =1\,,
\end{align*}
where the strict inequality holds because, as established above, $[T\mathbb{I}_{\{y\}}](x)>0$ and $g(y)=[(T\vert_{A^c})^{m}1](y)<1$; where the second inequality follows because $g(z)=0$ for all $z\in A$, because $g(z)\leq 1$ for all $z\in A^c$ due to Lemma~\ref{lemma:basic_properties_restriction}, and because $[T\mathbb{I}_{\{z\}}](x)\geq 0$ for all $z\in\states$ because $T$ satisfies~\ref{transprop:bounds}; and where the final equality used that $\sum_{z\in\states}\mathbb{I}_{\{z\}}=1$ together with the fact that $T$ is a linear map that satisfies~\ref{transprop:bounds}.

Hence we find that $[(T\vert_{A^c})^{m+1}1](x)<1$ for all $x\in A^c$ with $n_x=m+1$. Because we already established that $[(T\vert_{A^c})^{m+1}1](x)<1$ for all $x\in A^c$ with $n_x\leq m$, we conclude that, indeed, $[(T\vert_{A^c})^{m+1}1](x)<1$ for all $x\in A^c$ with $n_x\leq m+1$. This concludes the proof of the induction step, and hence our induction argument implies that, for any $m\in\nats$, $[(T\vert_{A^c})^{m}1](x)<1$ for all $x\in A^c$ with $n_x\leq m$. Because $n=\max_{x\in A^c}n_x$ satisfies $n\geq n_x$ for all $x\in A^c$, it follows that, as claimed, for all $k\geq n$ we have $n_x\leq n\leq k$ and hence $[(T\vert_{A^c})^{k}1](x)<1$ for all $x\in A^c$.
\qed
\end{proof}

\begin{corollary}\label{cor:norm_restriction_eventually_small_enough}
Let $T:\gambles\to\gambles$ be a linear map that satisfies~\ref{coh:nonneghomogen}-\ref{coh:bounds} and the reachability condition~\ref{reach_condition}. Then there is some $n\in\nats$ such that $\smash{\norm{(T\vert_{A^c})^n}<1}$.
\end{corollary}
\begin{proof}
By Lemma~\ref{lemma:complement_eventually_moves}, there is some $n\in\nats$ such that $[(T\vert_{A^c})^n1](x)<1$ for all $x\in A^c$. Since $0\leq (T\vert_{A^c})^n1$ due to Lemma~\ref{lemma:basic_properties_restriction}, this implies that $\norm{(T\vert_{A^c})^n1}<1$. It therefore suffices to show that $\smash{\norm{(T\vert_{A^c})^n}\leq \norm{(T\vert_{A^c})^n1}}$.

Consider any $f\in\reals^{A^c}$ such that $\norm{f}\leq 1$. Then $-1\leq f\leq 1$ due to the definition of $\norm{f}$. Using the linear character of $T\vert_{A^c}$ (and therefore of $(T\vert_{A^c})^n$) it follows that
\begin{equation*}
-(T\vert_{A^c})^n1 = (T\vert_{A^c})^n(-1) \leq (T\vert_{A^c})^nf \leq (T\vert_{A^c})^n1\,,
\end{equation*}
where the inequalities use the monotonicity properties of $(T\vert_{A^c})^n$ established in Lemma~\ref{lemma:basic_properties_restriction}. This implies that $\norm{(T\vert_{A^c})^nf}\leq \norm{(T\vert_{A^c})^n1}$. Because this is true for every $f\in\reals^{A^c}$ such that $\norm{f}\leq 1$, it follows that
\begin{equation*}
\norm{(T\vert_{A^c})^n} = \sup\bigl\{ \norm{(T\vert_{A^c})^nf}\,:\,f\in\reals^{A^c},\norm{f}\leq 1\bigr\}\leq \norm{(T\vert_{A^c})^n1}\,,
\end{equation*}
which concludes the proof.
\qed


\end{proof}

The next result summarizes the results established above, and will be crucial for the remainder of this paper. 
\begin{lemma}\label{lemma:submatrix_powers_vanish}
Let $\underline{T}:\gambles\to\gambles$ be a map that satisfies~\ref{coh:nonneghomogen}-\ref{coh:bounds} and the reachability condition~\ref{reach_condition}, and let $\mathcal{T}$ be its dual representation. For any $T\in\mathcal{T}$, it holds that $\lim_{k\to+\infty} (T\vert_{A^c})^k=0$ and, moreover, that $(I-T\vert_{A^c})^{-1} = \sum_{k=0}^{+\infty}(T\vert_{A^c})^k$, where $I$ is the identity map on $\reals^{ A^c}$.
\end{lemma}
\begin{proof}
Fix any $T\in\mathcal{T}$. Then it follows from Proposition~\ref{prop:dual_set_representation} and Lemma~\ref{lemma:lower_reachability_is_precise_reachability} that $T$ is a linear map from $\gambles$ to $\gambles$ that satisfies~\ref{coh:nonneghomogen}-\ref{coh:bounds} and~\ref{reach_condition}. By Corollary~\ref{cor:norm_restriction_eventually_small_enough}, this implies that there is some $n\in\nats$ such that $\norm{(T\vert_{A^c})^n}<1$, which implies that also $\smash{\norm{(T\vert_{A^c})^n}^{\frac{1}{n}}<1}$. Because $T\vert_{A^c}$ is a linear map (due to Definition~\ref{def:restriction}) that is bounded by Corollary~\ref{cor:lower_trans_is_bounded} and Lemma~\ref{lemma:restriction_bounded_is_bounded}, Lemma~\ref{lemma:bound_spectral_radius} therefore implies that $\rho(T\vert_{A^c})\leq \smash{\norm{(T\vert_{A^c})^n}^{\frac{1}{n}}}<1$.

By Corollary~\ref{cor:powers_vanish} this implies that $\lim_{k\to+\infty}(T\vert_{A^c})^k=0$ and, by Lemma~\ref{lemma:resolvent_is_series}, that $\sum_{k=0}^{+\infty} (T\vert_{A^c})^k=R(1,T\vert_{A^c})=(I-T\vert_{A^c})^{-1}$.
\qed
\end{proof}


Finally, we need the following result, which identifies the solution of the linear version of the system in Equation~\eqref{eq:fundamental_system}.
\begin{proposition}\label{prop:solution_of_linear_system}
Let $\underline{T}:\gambles\to\gambles$ be a map that satisfies~\ref{coh:nonneghomogen}-\ref{coh:bounds} and the reachability condition~\ref{reach_condition}, and let $\mathcal{T}$ be its dual representation. Then, for all $T\in\mathcal{T}$, there exists a unique solution $h_T\in\gambles$ of the linear system $h_T=\mathbb{I}_{A^c}+\mathbb{I}_{A^c}\cdot Th_T$. 
This $h_T$ is non-negative, and in particular, $h_T\vert_A=0$ and $h_T\vert_{A^c}=(I-T\vert_{A^c})^{-1}1$.
\end{proposition}
\begin{proof}	
Let $h_T$ be a vector such that $h_T\vert_A\coloneqq 0$ and $h_T\vert_{A^c}\coloneqq (I-T\vert_{A^c})^{-1}1$, where we note that $(I-T\vert_{A^c})^{-1}$ exists by Lemma~\ref{lemma:submatrix_powers_vanish}. Then $h_T\in\gambles$; specifically, $h_T(x)$ is finite for all $x\in\states$, because $(I-T\vert_{A^c})^{-1}$ is the inverse of a bounded linear map and thus bounded itself.

It is easily verified that $h_T$ satisfies the linear system of interest; for any $x\in A$ we trivially have
\begin{equation*}
h_T(x) = 0 = \mathbb{I}_{A^c}(x) + \mathbb{I}_{A^c}(x)\cdot[Th_T](x)\,.
\end{equation*}
Conversely, on $A^c$ it holds that $h_T\vert_{A^c} = (I-T\vert_{A^c})^{-1}1$, so, multiplying both sides with $(I-T\vert_{A^c})$, we obtain $h_T\vert_{A^c} = 1 + T\vert_{A^c}(h_T\vert_{A^c})$.
We note that, because $h_T\vert_A=0$ it holds that $h_T\cdot\mathbb{I}_{A^c}=h_T$, and hence
\begin{equation*}
T\vert_{A^c}(h_T\vert_{A^c}) = \bigl(T(h_T\cdot\mathbb{I}_{A^c})\bigr)\vert_{A^c} = \bigl(Th_T\bigr)\vert_{A^c}\,,
\end{equation*}
using Definition~\ref{def:restriction}, and so
\begin{equation*}
h_T\vert_{A^c} = 1 + \bigl(Th_T\bigr)\vert_{A^c} = \mathbb{I}_{A^c}\vert_{A^c} + \mathbb{I}_{A^c}\vert_{A^c}\cdot\bigl(Th_T\bigr)\vert_{A^c}\,. 
\end{equation*}
This establishes that $h_T$ solves the linear system of interest. 

Let us prove that $h_T$ is non-negative. Let $c\coloneqq \min_{y\in\states}h_T(y)$, and let $y\in \states$ be such that $h_T(y)=c$. Then $c\in\reals$ because $h_T\in\gambles$. Now suppose \emph{ex absurdo} that $y\in A^c$. Then, since $h_T-c\geq 0$, and because, as established above, $h_T\vert_{A^c} = 1 + \bigl(Th_T\bigr)\vert_{A^c}$, it follows that
\begin{align*}
h_T\vert_{A^c}-c= 1 + (Th_T)\vert_{A^c}-c = 1 + \bigl(T(h_T-c)\bigr)\big\vert_{A^c} \geq 1 + \bigl(T(0)\bigr)\big\vert_{A^c} = 1\,,
\end{align*}
where the second equality used Property~\ref{transprop:constant_additive} and where the inequality used Property~\ref{transprop:monotone}, which we can do because $T$ satisfies~\ref{coh:nonneghomogen}-\ref{coh:bounds} due to Proposition~\ref{prop:dual_set_representation}; and where the final equality used the linearity of $T$. Because $y\in A^c$, this implies in particular that $0=h_T(y)-c\geq 1$. From this contradiction, it follows that $y\in A$. Since $h_T(x)=0$ for all $x\in A$, this implies that $c=h_T(y)=0$, which means that $h_T$ is non-negative.

It remains to establish that $h_T$ is the unique solution in $\gambles$. To this end, let $g\in\gambles$ be any real-valued solution of this system, i.e. suppose that also $g=\mathbb{I}_{A^c}+\mathbb{I}_{A^c}\cdot Tg$. Then some algebra analogous to the above yields that $g\vert_A=0$ and $g\vert_{A^c}=(I-T\vert_{A^c})^{-1}1$, thus $g=h_T$, whence the solution is unique. 
\qed
\end{proof}

We can now prove our first main result, which was stated in Section~\ref{sec:introduction}.

\quad\newline
\noindent\emph{Proof of Proposition~\ref{prop:existence_is_unique_and_finite}}.\,\,
Let $\underline{h}$ be the minimal non-negative solution to~\eqref{eq:fundamental_system}; this solution exists by Proposition~\ref{prop:basic_existence}, although it could be in $(\reals\cup\{+\infty\})^{\states}$. Let us first establish that in fact $\underline{h}\in\gambles$. To this end, we note that by Proposition~\ref{prop:solutions_are_bounds_and_reached}, there is some $T\in\mathcal{T}$ such that $\underline{h}=h_T$, where $h_T$ is the minimal non-negative solution of the linear system $h_T=\mathbb{I}_{A^c}+\mathbb{I}_{A^c}\cdot T h_T$. 

Because $\underline{T}$ satisfies ~\ref{coh:nonneghomogen}-\ref{coh:bounds} and~\ref{reach_condition}, it follows from Proposition~\ref{prop:solution_of_linear_system} that the linear system $h_T=\mathbb{I}_{A^c}+\mathbb{I}_{A^c}\cdot T h_T$ has a unique solution in $\gambles$ that is non-negative, which equals the minimal non-negative solution $h_T$ due to Corollary~\ref{cor:unique_solution_is_minimal}. Hence $h_T\in\gambles$, and because $\underline{h}=h_T$, it follows that also $\underline{h}\in\gambles$.

%

To establish that the solution of~\eqref{eq:fundamental_system} is unique in $\gambles$, let $h_T$ and $h_S$ be any two solutions of~\eqref{eq:fundamental_system} in $\gambles$; without further assumptions, we will show that $h_T\leq h_S$. 
First, for any $x\in A$, it holds that 
$h_T(x) = \mathbb{I}_{A^c}(x) + \mathbb{I}_{A^c}(x)\cdot [\underline{T}\,h_T](x) = 0$, 
and, by the same argument, $h_S(x)=0$, so $h_T$ and $h_S$ agree on $A$.

Next, due to Proposition~\ref{prop:dual_set_representation}, there are $T,S\in\mathcal{T}$ such that $\underline{T}\,h_T=Th_T$ and $\underline{T}\,h_S=Sh_S$.
It also follows from Proposition~\ref{prop:dual_set_representation} that 
$Th_T = \underline{T}\,h_T \leq Sh_T$, 
because $S\in\mathcal{T}$. In turn, and using that $h_T(x)=0$ for all $x\in A$, and that therefore $h_T\cdot\mathbb{I}_{A^c}=h_T$, this implies that
\begin{equation*}
T\vert_{A^c}(h_T\vert_{A^c}) = \bigl(Th_T\bigr)\vert_{A^c} \leq \bigl(Sh_T\bigr)\vert_{A^c} = S\vert_{A^c}(h_T\vert_{A^c})\,.
\end{equation*}

Substituting this inequality into the system~\eqref{eq:fundamental_system} restricted to $A^c$ yields
\begin{align*}
h_T\vert_{A^c} = 1 + T\vert_{A^c}\,(h_T\vert_{A^c}) \leq 1 + S\vert_{A^c}\,(h_T\vert_{A^c})\,.
\end{align*}
Now note that $S\vert_{A^c}$ is a monotone map due to Proposition~\ref{prop:dual_set_representation} and Lemma~\ref{lemma:basic_properties_restriction}. Hence, and using the fact that $S\vert_{A^c}$ is a linear map, expanding this inequality into the right-hand side gives, after $n\in\nats$ expansions,
\begin{equation*}
h_T\vert_{A^c} \leq (S\vert_{A^c})^n(h_T\vert_{A^c}) + \sum_{k=0}^{n-1} (S\vert_{A^c})^k1\,.
\end{equation*}
Taking limits in $n$, and using that $\lim_{n\to+\infty}(S\vert_{A^c})^n=0$ and $\sum_{k=0}^{+\infty} (S\vert_{A^c})^k=(I-S\vert_{A^c})^{-1}$ due to Lemma~\ref{lemma:submatrix_powers_vanish}, it follows that
\begin{equation*}
h_T\vert_{A^c} \leq (I-S\vert_{A^c})^{-1}1 = h_S\vert_{A^c}\,,
\end{equation*}
where we used Proposition~\ref{prop:solution_of_linear_system} for the equality, which we can do because, since $h_S$ satisfies Equation~\eqref{eq:fundamental_system} by assumption and since $\underline{T}\,h_S=Sh_S$ by the selection of $S\in\mathcal{T}$, it holds that  $h_S=\mathbb{I}_{A^c}+\mathbb{I}_{A^c}\cdot S\,h_S$. Hence we have found that $h_T\vert_{A^c}\leq h_S\vert_{A^c}$. However, we can now repeat the above argument, \emph{mutatis mutandis}, to establish that also $h_S\vert_{A^c}\leq h_T\vert_{A^c}$. Hence we conclude that $h_T$ and $h_S$ agree on $A^c$. Because we already established that they agree on $A$, we conclude that $h_T=h_S$, and that therefore Equation~\eqref{eq:fundamental_system} has a unique solution in $\gambles$. Because $\underline{h}$ is \emph{a} solution to Equation~\eqref{eq:fundamental_system} in $\gambles$, we conclude that $\underline{h}$ must be the unique solution. Because $\underline{h}$ is non-negative, it follows that the unique solution in $\gambles$ is non-negative.
\qed

\section{A Computational Method}\label{sec:computational_method}

The computational method that we propose in Proposition~\ref{prop:algorithm_works} works by iterating over specific choices of the extreme points of the dual representation $\mathcal{T}$ of $\underline{T}$. The next result establishes that these extreme points exist and that $\underline{T}\,(\cdot)$ obtains its value in an extreme point of $\mathcal{T}$; this is well-known, so we omit the proof.
\begin{lemma}\label{lemma:extreme_points}
Let $\underline{T}:\gambles\to\gambles$ be a map that satisfies~\ref{coh:nonneghomogen}-\ref{coh:bounds}, and let $\mathcal{T}$ be its dual representation. Then $\mathcal{T}$ has a non-empty set of extreme points. Moreover, for all $f\in\gambles$, there is an extreme point $T$ of $\mathcal{T}$ such that $\underline{T}\,f=Tf$.
\end{lemma}

We can now state the second main result of this work.
\begin{proposition}\label{prop:algorithm_works}
Let $\underline{T}:\gambles\to\gambles$ be a map that satisfies~\ref{coh:nonneghomogen}-\ref{coh:bounds} and the reachability condition~\ref{reach_condition}, and let $\mathcal{T}$ be its dual representation. Let $T_1\in\mathcal{T}$ be any extreme point of $\mathcal{T}$ and, for all $n\in\nats$, let $h_n$ be the unique solution in $\gambles$ of the linear system $h_n=\mathbb{I}_{A^c}+\mathbb{I}_{A^c}\cdot T_n h_n$, and let $T_{n+1}$ be an extreme point of $\mathcal{T}$ such that $\underline{T}h_{n}=T_{n+1}h_n$. 
Then the sequence $\{h_n\}_{n\in\nats}$ is non-increasing, and its limit $h_*\coloneqq \lim_{n\to+\infty}h_n$ is the unique solution of~\eqref{eq:fundamental_system} in $\gambles$.
\end{proposition}
\begin{proof}
First, for any $n\in\nats$, the solution $h_n\in\gambles$ of the linear system $h_n=\mathbb{I}_{A^c}+\mathbb{I}_{A^c}\cdot T_nh_n$ exists and is non-negative due to Proposition~\ref{prop:solution_of_linear_system}, and there is an extreme point $T_{n+1}$ of $\mathcal{T}$ that satisfies $\underline{T}\,h_n=T_{n+1}h_n$ due to Lemma~\ref{lemma:extreme_points}. Hence the sequence $\{h_n\}_{n\in\nats}$ is well-defined and bounded below by $0$.
Due to Proposition~\ref{prop:solution_of_linear_system}, it holds that $h_n(x)=0$ for all $x\in A$ and all $n\in\nats$. Thus the sequence $\{h_n\}_{n\in\nats}$ is trivially non-increasing and convergent on $A$. We will next establish the same on $A^c$.

First note that $h_n\cdot\mathbb{I}_{A^c}=h_n$ for all $n\in\nats$, because $h_n(x)=0$ for all $x\in A$. Now fix any $n\in\nats$. Then $T_{n+1}h_n = \underline{T}\,h_n\leq T_nh_n$ because $T_n\in\mathcal{T}$, and therefore in particular it holds that
\begin{equation*}
T_{n+1}\vert_{A^c}(h_n\vert_{A^c}) = \bigl(T_{n+1}h_n\bigr)\vert_{A^c} \leq \bigl(T_{n}h_n\bigr)\vert_{A^c} = T_{n}\vert_{A^c}(h_n\vert_{A^c})\,.
\end{equation*}
Therefore, and because $h_n=\mathbb{I}_{A^c}+\mathbb{I}_{A^c}\cdot T_nh_n$, it holds that
\begin{equation*}
h_n\vert_{A^c} = 1 + T_n\vert_{A^c}(h_n\vert_{A^c}) \geq 1 + T_{n+1}\vert_{A^c}(h_n\vert_{A^c})\,.
\end{equation*}
Now note that $T_{n+1}\vert_{A^c}$ is a monotone map due to Proposition~\ref{prop:dual_set_representation} and Lemma~\ref{lemma:basic_properties_restriction}. Hence, and
using the fact that $T_{n+1}\vert_{A^c}$ is a linear map, repeatedly expanding this inequality into the right-hand side gives, after $m\in\nats$ expansions,
\begin{equation*}
h_n\vert_{A^c} \geq \bigl(T_{n+1}\vert_{A^c}\bigr)^m(h_n\vert_{A^c}) + \sum_{k=0}^{m-1}\bigl(T_{n+1}\vert_{A^c}\bigr)^k1\,.
\end{equation*}
Taking limits in $m$, and using $\lim_{m\to+\infty}(T_{n+1}\vert_{A^c})^m=0$ and $\sum_{k=0}^{+\infty} (T_{n+1}\vert_{A^c})^k=(I-T_{n+1}\vert_{A^c})^{-1}$ due to Lemma~\ref{lemma:submatrix_powers_vanish}, it follows that
\begin{equation*}
h_n\vert_{A^c} \geq (I-T_{n+1}\vert_{A^c})^{-1}1 = h_{n+1}\vert_{A^c}\,,
\end{equation*}
where we used Proposition~\ref{prop:solution_of_linear_system} for the equality, which we can do because, by construction, $h_{n+1}\in\gambles$ and $h_{n+1}=\mathbb{I}_{A^c}+\mathbb{I}_{A^c}\cdot T_{n+1}h_{n+1}$. Thus the sequence $\{h_n\}_{n\in\nats}$ is non-increasing. Because we know that $h_n\geq 0$ for all $n\in\nats$ due to Proposition~\ref{prop:solution_of_linear_system}, it follows that the limit $h_*\coloneqq \lim_{n\to+\infty} h_n$ exists.

Let us now show that $h_*$ solves~\eqref{eq:fundamental_system}. To this end, fix $n\in\nats$ and note that
\begin{align*}
\lVert h_* - \mathbb{I}_{A^c}-\mathbb{I}_{A^c}\cdot\underline{T}\,h_*\rVert 
  &\leq \lVert h_* - h_{n+1}\rVert 
  + \lVert h_{n+1}-\mathbb{I}_{A^c}-\mathbb{I}_{A^c}\cdot T_{n+1}h_{n+1}\rVert \\
  &\quad\quad\quad 
  + \lVert \mathbb{I}_{A^c}\cdot T_{n+1}h_{n+1} - \mathbb{I}_{A^c}\cdot \underline{T}\,h_*\rVert \\
  &= \lVert h_* - h_{n+1}\rVert 
   + \lVert \mathbb{I}_{A^c}\cdot T_{n+1}h_{n+1} - \mathbb{I}_{A^c}\cdot \underline{T}\,h_*\rVert \\
  &\leq \lVert h_* - h_{n+1}\rVert 
      + \lVert T_{n+1}h_{n+1} - \underline{T}\,h_* \rVert \\
  &\leq \lVert h_* - h_{n+1}\rVert 
      + \lVert T_{n+1}h_{n+1} - \underline{T}\,h_{n}\rVert 
      + \lVert \underline{T}\,h_{n} - \underline{T}\,h_*\rVert \\
&\leq \lVert h_* - h_{n+1}\rVert 
  + \lVert T_{n+1}h_{n+1} - \underline{T}\,h_{n}\rVert 
  + \lVert \,h_{n} - \,h_*\rVert \\
&= \lVert h_* - h_{n+1}\rVert 
+ \lVert T_{n+1}h_{n+1} - T_{n+1}h_{n}\rVert 
+ \lVert \,h_{n} - \,h_*\rVert \\
&\leq  \lVert h_* - h_{n+1}\rVert 
+ \lVert h_{n+1} - h_{n}\rVert 
+ \lVert \,h_{n} - \,h_*\rVert\,,
\end{align*}
where we used property~\ref{transprop:nonexpand} for the final two inequalities.
Taking limits in $n$, all summands on the right-hand side vanish, from which we conclude that $h_* = \mathbb{I}_{A^c}+\mathbb{I}_{A^c}\cdot\underline{T}\,h_*$. In other words, $h_*$ is a solution of~\eqref{eq:fundamental_system}. Because $\underline{T}$ satisfies~\ref{coh:nonneghomogen}-\ref{coh:bounds} and~\ref{reach_condition}, by Proposition~\ref{prop:existence_is_unique_and_finite}, Equation~\eqref{eq:fundamental_system} has a unique solution $\underline{h}$, and hence we have that $\underline{h}=h_*$.
\qed
\end{proof}
We remark without proof that if, in the statement of Proposition~\ref{prop:algorithm_works}, we instead take each $T_{n+1}$ to be an extreme point of $\mathcal{T}$ such that $\overline{T}h_n=T_{n+1}h_n$, then the sequence instead becomes non-decreasing, and converges to a limit that is the unique non-negative solution of~\eqref{eq:fundamental_system_conjugate} in $\gambles$. So, we can use a completely similar method to compute \emph{upper} expected hitting times. That said, in the remainder of this work, we will focus on the version for lower expected hitting times.

\section{Complexity Analysis}\label{sec:complexity}

We will say that a sequence $\{h_n\}_{n\in\nats}$ of vectors is \emph{strictly decreasing} if $h_{n+1}\leq h_{n}$ and $h_{n+1}\neq h_{n}$ for all $n\in\nats$.
\begin{corollary}\label{cor:sequence_decreasing_or_done}
Let $\underline{T}:\gambles\to\gambles$ be a map that satisfies~\ref{coh:nonneghomogen}-\ref{coh:bounds} and~\ref{reach_condition}, and let $\mathcal{T}$ be its dual representation. Then any sequence $\{h_n\}_{n\in\nats}$ in $\gambles$ constructed as in Proposition~\ref{prop:algorithm_works}, is either strictly decreasing everywhere, or, for some $m\in\nats$,  is strictly decreasing for all $n<m$ and satisfies $h_k=h_*$ for all $k\geq m$, with $h_*=\lim_{n\to+\infty} h_n$.
\end{corollary}
\begin{proof}
Because we know from Proposition~\ref{prop:algorithm_works} that the sequence $\{h_n\}_{n\in\nats}$ is non-increasing, the converse of this statement is that there is some $m\in\nats$ such that $h_m=h_{m+1}$ with $h_m\neq h_*$. Suppose \emph{ex absurdo} that this is the case.

Consider the co-sequence $\{T_n\}_{n\in\nats}$ of extreme points of $\mathcal{T}$ that was used to construct the sequence $\{h_n\}_{n\in\nats}$. Let $T\coloneqq T_{m+1}$ and $h\coloneqq h_m=h_{m+1}$. Then it holds that $\underline{T}\, h = Th$, and $h$ is the unique solution in $\gambles$ of the linear system $h=\mathbb{I}_{A^c}+\mathbb{I}_{A^c}\cdot Th$. 
Following the conditions of Proposition~\ref{prop:algorithm_works}, we now construct a sequence $\{T_n'\}_{n\in\nats}$, and a co-sequence $\{h_n'\}_{n\in\nats}$ in $\gambles$, such that $h_n'=\mathbb{I}_{A^c}+\mathbb{I}_{A^c}\cdot T_n'h_n'$ for all $n\in\nats$. To this end, first set $T_1'\coloneqq T$. This yields $h_1' = h$, and because $T\in\mathcal{T}$ is an extreme point of $\mathcal{T}$ that satisfies $\underline{T}\,h_1'=\underline{T}\,h=Th=Th_1'$, we can take $T_2'\coloneqq T$, which yields $h_2'=h$. Proceeding in this fashion, we obtain $T_n'=T$ and $h_n'=h$ for all $n\in\nats$. This sequence $\{h_n'\}_{n\in\nats}$ satisfies the conditions of Proposition~\ref{prop:algorithm_works}, but $\lim_{n\to+\infty}h_n'=h=h_m\neq h_*$, a contradiction.
\qed
\end{proof}

\begin{corollary}\label{cor:only_the_extreme_points}
Let $\underline{T}:\gambles\to\gambles$ be a map that satisfies~\ref{coh:nonneghomogen}-\ref{coh:bounds} and~\ref{reach_condition}, and let $\mathcal{T}$ be its dual representation. Suppose that $\mathcal{T}$ has a finite number $m\in\nats$ of extreme points. Then any sequence $\{h_n\}_{n\in\nats}$ in $\gambles$ constructed as in Proposition~\ref{prop:algorithm_works}, satisfies $h_k=h_*=\lim_{\ell\to+\infty}h_\ell$ for all $k\geq n$, for some $n\leq m$.
\end{corollary}
\begin{proof}
Because $\mathcal{T}$ has $m\in\nats$ extreme points, it follows from Proposition~\ref{prop:algorithm_works} that the elements of the sequence $\{h_n\}_{n\in\nats}$ can take at most $m$ distinct values; each corresponding to an extreme point of $\mathcal{T}$. Hence there is some $n\leq m$ such that $h_n=h_{m+1}$, that is, $h_{m+1}$ must have a value that was already obtained earlier in the sequence. Now apply Corollary~\ref{cor:sequence_decreasing_or_done}.
\qed
\end{proof}

This result suggests that the numerical scheme proposed in Proposition~\ref{prop:algorithm_works} takes at most $m$ iterations, where $m$ is the number of extreme points of $\mathcal{T}$. In fact, it is possible to show that this bound is tight, i.e. for any $m\in\nats$ there is a map $\underline{T}$ whose dual representation $\mathcal{T}$ has $m$ extreme points, for which the method suggested in Proposition~\ref{prop:algorithm_works} takes exactly $m$ iterations when $T_1\in\mathcal{T}$ is chosen carefully. Unfortunately, the construction that shows the tightness is fairly involved, and we must omit it here for reasons of brevity.

Let us now remark on the other computational aspects of the algorithm. At each step $n\in\nats$ of the method, we need to solve two problems. First, given $T_{n}\in\mathcal{T}$, we need to compute $h_n$. This is equivalent to computing the expected hitting time of $A$ for the precise Markov chain identified by $T_n$, so this step can be solved by any method available in the literature for the latter problem. However, Proposition~\ref{prop:solution_of_linear_system} tells us that $h_n$ can be obtained by solving a linear system. Hence, ignoring issues of numerical precision, the complexity of this step is at most $\mathcal{O}(\lvert\states\rvert^\omega)$~\cite{ibarra1982generalization}\label{notation:complexity}, which is the complexity of multiplying two $\abs{\states}\times\abs{\states}$ matrices; the current best estimate for this exponent is $\omega\approx 2.38$~\cite{Brand2020ADL}. 

Secondly, we need to find an extreme point $T_{n+1}\in\mathcal{T}$ such that $T_{n+1}h_n=\underline{T}\,h_n=\inf_{T\in\mathcal{T}}Th_n$. The complexity of this step will depend strongly on the way $\underline{T}$ is encoded. In many practical applications, however, it will be derived from a given set $\mathcal{T}$ of linear maps satisfying~\ref{coh:nonneghomogen}-\ref{coh:bounds}; that is, the dual representation $\mathcal{T}$ is often given explicitly and used to describe $\underline{T}$, rather than the other way around. Since $\mathcal{T}$ has separately specified rows, to identify $\mathcal{T}$ one only has to specify the sets $\mathcal{T}_x$. In turn, these sets $\mathcal{T}_x$ can, in practice, often be described by a finite number of linear inequalities; in the context of Markov chains, these inequalities represent given bounds on the transition probabilities for moving from the state $x$ to other states. Under these conditions, each $\mathcal{T}_x$ will be a set that is non-empty (assuming feasibility of the specified constraints), closed, and convex, with a finite number of extreme points. 

In this case, computing each $[\underline{T}\,f](x)$ (and in particular $[\underline{T}\,h_n](x)$) reduces to solving a linear programming problem in $N=\lvert\states\rvert+c$ variables,\footnote{It is worth noting that this analysis is somewhat pessimistic, because it assumes no real structure on the underlying dynamical system; in practice it will often be unlikely that a given state $x$ can move to \emph{every} other state $y$ in a single step, whence $\mathcal{T}_x$ will live in a subspace of dimension (much) less than $\lvert\states\rvert$.} where $c$ is the number of constraints used to specify $\mathcal{T}_x$. For instance, the simplex algorithm can be used to solve this problem, with the added benefit that the returned optimal solution is an extreme point of $\mathcal{T}_x$. Unfortunately, the complexity of the simplex algorithm depends on the pivot rule that is used, and the existence of a polynomial-time pivot rule is still an open research question~\cite{todd2002linearprog,dadush2017smoothedsimplex}. Nevertheless, it typically performs very well in practice; roughly speaking, probabilistic analyses typically yield an expected number of iterations on the order of $\mathcal{O}(N^{2+\alpha})$, with $\alpha>0$ depending on the specifics of the analysis. For example,~\cite{dadush2017smoothedsimplex} gives a bound for the expected number of iterations on the order of $\mathcal{O}(\abs{\states}^2\sqrt{\log c}\sigma^{-2}+\abs{\states}^3\log^{\frac{2}{3}}c)$, where $\sigma$ is a parameter of the probabilistic analysis. Crucially, however, we are not aware of a result that shows an (expected) runtime of the simplex algorithm that is sub-quadratic in $\abs{\states}$. Moreover, a very recent result~\cite{Brand2020ADL} gives a deterministic interior-point algorithm that solves this problem (approximately), and which runs in $\tilde{\mathcal{O}}(N^\omega)$,\footnote{Following~\cite{Brand2020ADL}, $\tilde{\mathcal{O}}(N)$ hides $\mathrm{polylog}(N)$ factors.} where, as before, $\mathcal{O}(N^\omega)$ is the complexity of multiplying two $N\times N$ matrices. As that author notes, this complexity can be regarded as essentially optimal. However, because this is an interior-point method, the resulting solution is not guaranteed to be an extreme point of $\mathcal{T}_x$, which we require for our algorithm, so we cannot use the result here other than for illustrating the complexity of the minimisation problem. 

In any case, the above considerations suggest that finding $T_x\in\mathcal{T}_x$ such that $T_xf=[\underline{T}\,f](x)$, for a generic $\mathcal{T}_x$ that is specified using a finite number of linear constraints, will typically take \emph{at least} $\mathcal{O}(\abs{\states}^2)$ time.
This then needs to be repeated $\lvert\states\rvert$ times, once for each $x\in\states$. Hence, finding an extreme point $T_{n+1}\in\mathcal{T}$ such that $\underline{T}\,h_n=T_{n+1}h_n$ can be done in polynomial expected time, but requires at least $\mathcal{O}(\lvert\states\rvert^3)$ time (in expectation).
In conclusion, the above analysis tells us that a single step of our proposed method has a expected runtime complexity of at least $\mathcal{O}(\lvert\states\rvert^3)$, where the complexity is dominated by finding $T_{n+1}\in\mathcal{T}$ such that $\underline{T}\,h_n=T_{n+1}h_n$, rather than by computing $h_n$ from the previous $T_n$. 

Let us finally compare this to the only other method in the literature of which we are aware, which is described in Proposition~\ref{prop:solutions_by_iteration}. Starting with $\underline{h}_0\coloneqq \mathbb{I}_{A^c}$, at each step of this method, we need to compute $\underline{T}\,\underline{h}_{n-1}$. Following the discussion above, and depending on the algorithm and analysis used, we can expect this to have a complexity of at least $\mathcal{O}(\abs{\states}^3)$. Since this was also the dominating factor in the complexity of our new algorithm, we conclude that the two methods have comparable complexity (order) per iteration. What remains is a comparison of the number of iterations they require. With $\underline{h}_n$ as in Proposition~\ref{prop:solutions_by_iteration}, it follows from~\cite[Lemma 42]{krak2019extendedhitting} that $\underline{h}_n\leq n+1$ for all $n\in\nats$. Hence, for any $\epsilon>0$, this method will take at least $(\norm{\underline{h}}-\epsilon)-1$ iterations, before $\norm{\underline{h}_n-\underline{h}}\leq \epsilon$.
The new method from Proposition~\ref{prop:algorithm_works} is thus more efficient than the method from Proposition~\ref{prop:solutions_by_iteration}, when the number $m$ of extreme points of $\mathcal{T}$ is less than $\norm{\underline{h}}$. Of course, since we do not know $\underline{h}$ to begin with, this does not provide a practically useful way to determine which method to use in practice.

Moreover, because the number $m$ of extreme points of $\mathcal{T}$ is potentially extremely large, one may wonder whether the condition $m<\norm{\underline{h}}$ is ever really satisfied in practice. However, a first empirical analysis suggests that the bound $m$ on the number of iterations, although tight, is not representative of the average case complexity of our algorithm.

In particular, we performed experiments where we randomly created sets $\mathcal{T}$, and applied the method from Proposition~\ref{prop:algorithm_works} to the resulting systems. In these experiments, we varied the size of $\states$ between $10^2$ and $10^3$, while keeping $A$ a singleton throughout. For each setting, we generated sets $\mathcal{T}_x$ by sampling 50 points uniformly at random in the space of probability mass functions on $\states$, and taking $\mathcal{T}_x$ to be the convex hull of these points. Thus each $\mathcal{T}_x$ has 50 extreme points (almost surely with respect to the uniform sampling), and the set $\mathcal{T}$ constructed from them has $50^{\lvert\states\rvert}$ extreme points. We then ran the method from Proposition~\ref{prop:algorithm_works}, where we selected $T_1$ such that $\overline{T}\,\mathbb{I}_{A}=T_1\mathbb{I}_{A}$, and the number of iterations until convergence was recorded. Using the notation from Proposition~\ref{prop:algorithm_works}, we say that it converged after $n>1$ iterations, when we observe $h_n=h_{n-1}$; this is valid by Corollary~\ref{cor:sequence_decreasing_or_done}. This was repeated 50 times for each value of $\lvert\states\rvert$. 

The results of these experiments are shown in Figure~\ref{fig:empirical_iterations}. We see that in the majority of cases, the method converges to the correct solution in three iterations, although sometimes it takes four iterations. There were no instances where the method took more iterations to converge, although in one instance (not shown) with $\lvert\states\rvert=10^3$, the method converged in only two iterations.

\begin{figure}[htb]
	\begin{tikzpicture}
	\begin{axis}[x tick label style = {/pgf/number format/1000 sep=},
	width=0.8\textwidth, height=3.5cm,
	ylabel=Count,enlargelimits=0.1,legend style={at={(1.05, 0.5)},anchor=west,legend columns=1},
	ybar,
	bar width=4pt,
	xmin=100, xmax=1000,
	xlabel={Number $\lvert\states\rvert$ of states},
	xtick={100,200,300,400,500,600,700,800,900,1000},
	ytick={0,25,50}]
	
	
	\addplot+[color=grijs!65, draw=grijs!65] coordinates{(100,40)(200,42)(300,46)(400,43)(500,46)(600,43)(700,44)(800,42)(900,48)(1000,45)};
	
	\addplot+[color=geel!65, draw=geel!65] coordinates{(100,10)(200,8)(300,4)(400,7)(500,4)(600,7)(700,6)(800,8)(900,2)(1000,4)};
	
	\legend{~3 Iterations~\quad,~4 Iterations~\quad}
	
	\end{axis}
	\end{tikzpicture}
	\caption{Number of instances, out of 50, on which the method from Proposition~\ref{prop:algorithm_works} converged in a given number of iterations, for different sizes of the set $\states$.}\label{fig:empirical_iterations}
\end{figure}
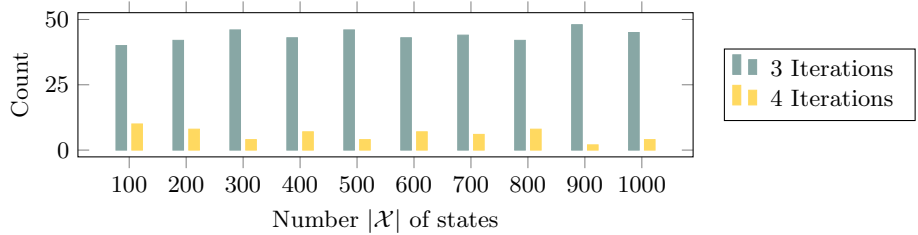

Although still preliminary, we believe that these results are indicative of an average runtime complexity that is vastly more efficient than the worst case suggested by Corollary~\ref{cor:only_the_extreme_points}. In future work, we hope to examine this average case performance more thoroughly. Finally, in~\cite{krak2019hitting} the authors already noted the close connection between the system~\eqref{eq:fundamental_system}, and the equations of optimality that one encounters in the theory of Markov decision processes (MDPs)~\cite{feinberg2012handbook}. Moreover, the method that we propose in Proposition~\ref{prop:algorithm_works}, although discovered independently, is reminiscent of the \emph{policy iteration} algorithm for MDPs. Interestingly, the method in Proposition~\ref{prop:solutions_by_iteration} bears a similar resemblance to the \emph{value iteration} algorithm for MDPs. We hope to explore these connections more fully in future work.

\subsubsection*{Acknowledgements.}
The main results of this work were previously presented, without proof, as an abstract-and-poster contribution at ISIPTA 2019. The author wishes to express his sincere gratitude to Jasper De Bock for stimulating discussions and insightful comments during the preparation of this manuscript. He would also like to thank an anonymous reviewer for their helpful suggestions.

%
%
%
%
\bibliographystyle{splncs04}
\bibliography{ComputingHittingTimes}

\end{document}